\documentclass[12pt]{amsart}
\usepackage[utf8]{inputenc} 
\usepackage[T1]{fontenc}
\usepackage{graphicx}
\usepackage{tikz, pgfplots}
\usepackage{amsfonts}
\usepackage{amssymb,amscd,latexsym}
\usepackage{amsmath}
\usepackage{epsfig}
\usepackage{color}
\usepackage[notref,notcite]{showkeys}

\newtheorem{thm}{Theorem}[section]
\newtheorem{cor}[thm]{Corollary}
\newtheorem{prop}[thm]{Proposition}

\newtheorem{defin}[thm]{Definition}

\newtheorem{rmk}[thm]{Remark}
\newtheorem{question}[thm]{Question}

\newtheorem{ex}[thm]{Example}

\newcommand{\OO}{{\mathcal O}}
\newcommand{\codim}{\operatorname{codim}}

\renewcommand{\Im}{\operatorname{Im}}
\def\p{\mathbb{P}}

\def\C{\mathbb{C}}

\def\P{\mathbb{P}}
\def\rk{\operatorname{rk}}

\def\Hess{\operatorname{Hess}}
\def\hess{\operatorname{hess}}

\newcommand{\Hilb}{\operatorname{Hilb}}

\newcommand{\Ann}{\operatorname{Ann}}

\newcommand{\Mac}{\operatorname{Mac}}

\newcommand{\ba}{\mathcal{B}}

\begin{document}

\title{Higher order Jacobians, Hessians and Milnor algebras}

\author[A. Dimca]{Alexandru Dimca$^1$}
\address{Universit\'e C\^ ote d'Azur, CNRS, LJAD and INRIA, France and Simion Stoilow Institute of Mathematics,
P.O. Box 1-764, RO-014700 Bucharest, Romania}
\email{dimca@unice.fr}

\author[R. Gondim]{Rodrigo Gondim$^2$}
\address{Universidade Federal Rural de Pernambuco, av. Don Manoel de Medeiros s/n, Dois Irmãos - Recife - PE
52171-900, Brasil}
\email{rodrigo.gondim@ufrpe.br}

\author[G. Ilardi]{Giovanna Ilardi}
\address{Dipartimento Matematica Ed Applicazioni ``R. Caccioppoli''
Universit\`{a} Degli Studi Di Napoli ``Federico II'' Via Cintia -
Complesso Universitario Di Monte S. Angelo 80126 - Napoli - Italia}
\email{giovanna.ilardi@unina.it}

\thanks{$^1$ This work has been partially supported by the French government, through the $\rm UCA^{\rm JEDI}$ Investments in the Future project managed by the National Research Agency (ANR) with the reference number ANR-15-IDEX-01 and by the Romanian Ministry of Research and Innovation, CNCS - UEFISCDI, grant PN-III-P4-ID-PCE-2016-0030, within PNCDI III}
\thanks{$^2$ The second author was partially supported by FACEPE ATP ATP-0005-1.01/18, by ICTP IDAN research in pairs 2018-2019 and by IMPA visiting fellowship Summer posdoc 2019.} 

\subjclass[2010]{Primary 13A02, 05E40; Secondary 13D40, 13E10}

\keywords{Jacobian ideal, Hessian, polar mapping, standard graded Artinian Gorenstein algebras, Hilbert functions, Lefschetz properties}

 \begin{abstract} 
We  introduce and study higher order Jacobian ideals,  higher order and mixed  Hessians, higher order polar maps, and higher order Milnor algebras associated to a reduced projective hypersurface. 
We relate these higher order objects to some standard graded Artinian Gorenstein algebras, and we study the corresponding Hilbert functions and Lefschetz properties. 
 
 \end{abstract}

\maketitle

\section{Introduction}

In Algebraic Geometry and Commutative Algebra, the {\it Jacobian ideal} of a homogeneous reduced form $f \in R=\C[x_0,\ldots,x_n]$, denoted by $J(f )= (\frac{\partial f}{\partial x_0}, \ldots, \frac{\partial f}{\partial x_0})$, plays several key roles. 
Let $X=V(f) \subset \P^n$ be the associated hypersurface in the projective space. 
The linear system associated to the Jacobian ideal defines the {\it polar map} $\varphi_X:\P^n \dashrightarrow \P^n$, also called the {\it gradient map}, whose image is the {\it polar image} of $X$, denoted by $Z_X = \overline{\varphi(\P^n)}$. The restriction of the polar map to the hypersurface is the {\it Gauss map} of $X$, $\mathcal{G}_X=\varphi_X|X$, whose image is the {\it dual variety} of $X$. 
The base locus of these maps is the singular scheme of the hypersurface $X$, see \cite{Ru}. The {\it Milnor algebra} of $f$, also called the {\it Jacobian ring} of $f$, is the quotient $M(f )= R/J(f)$. This graded algebra  is closely related to the Hodge filtration on the cohomology of $X$ and the period map, see \cite{DSa,Gr,Se}.    {\it The aim of this paper is to construct higher order versions of these 
classical objects, explicit some relations among them and extend some classical results to this higher order context.} 

In the second section we give some definitions of Artinian Gorenstein algebras, Hessians, Lefschetz properties, Jacobian ideals and Milnor algebras, and review some known results. The {\it standard Artinian Gorenstein algebra} $A(f)$, associated to $f$, is given by Macaulay Matlis duality: the ring $Q = \C[X_0,\ldots,X_n]$ acts on $R$ via the identification $X_i =\frac{\partial }{\partial x_i}$, and we define $A(f )= Q/\Ann(f)$, see \cite{MW}.
Since the Jacobian matrix associated to the polar map is the {\it Hessian matrix} of $f$, see \cite[Chapter 7]{Ru}, one gets that $\varphi_X$ is a dominant map, that is $Z _X = \P^n$, if and only if the {\it Hessian determinant} $\hess_f \ne 0$.  A description of the forms $f$ with $\hess_f = 0$ is given by the Gordan-Noether criterion, and can be found, for example, in \cite{CRS, GR, Go, GRu, Ru}. 
 The new results in the second sections are Proposition \ref{propSLP1} and Proposition \ref{propSLP2},
dealing with the Lefschetz properties of smooth cubic surfaces in $\P^3$ and  smooth quartic curves in $\P^2$.

In the third section we introduced the higher order Jacobian ideals $J^k(f)
$ and the corresponding Milnor algebras $M^k(f)$. The Gauss map $\mathcal{G}$ of a smooth hypersurface is a birational morphism, see for instance \cite{GH,Z}.
The natural $k$-th order version of smoothness is the hypothesis that all the points of $X$ have multiplicity at most $k$. In Theorem \ref{T1} we prove that this condition is equivalent to the $k$-th order Milnor algebra $M^k(f)$ being Artinian, generalizing a classical result for non singular hypersurfaces, and a second order result that can be found in \cite{DSt1}. We also discuss when the Hessian  of the form $f$ belongs to the Jacobian ideal $J(f)$, see Proposition \ref{P1.5} and Question \ref{questionHESS}.

In the fourth section we first show that the $k$-th order Milnor algebra $M^k(f)$ determines the hypersurface $V(f)$ up-to projective equivalence, for a generic $f$ and any
$k\leq d/2-1$, see Theorem \ref{thmZW2}. In the rather long Example \ref{ex3}, we look at quartic curves in $\P^2$, both smooth and singular, and we compute the Hilbert functions for our graded algebras $M^k(f)$ and $A(f)$
as well as the minimal resolutions as a graded $Q$-module for  $A(f)$ in many cases.

In the fifth section,
we construct the $k$-th polar map $\varphi^k_X$ and we prove, in Theorem \ref{thm:polarrankhess}, a higher order version of the Gordan-Noether criterion for the degeneracy of this $k$-th polar map. In this setting, we use the mixed Hessians developed in \cite{GZ}, generalizing higher order Hessians introduced in \cite{MW}. In Corollaries \ref{cor:polar1} and \ref{corHess10}, we give sufficient conditions for the non degeneracy of $\varphi^k_X$ and Theorem \ref{T2} give also some information about the degree  of the $k$-th polar map. 
In \cite{D}, the author showed that the natural higher order related dual map, $\psi^k_X = \varphi^k_X|X$, is a finite map. 
In Theorem \ref{T2}, we assume that the $k$-th Milnor algebra is Artinian to prove that $\varphi^k_X$ is finite. 

\bigskip

We would like to thank the referee for his very careful reading of our manuscript and for his suggestions which greatly improved the presentation of our results.

\section{Preliminaries}

\subsection{Artinian Gorenstein algebras and mixed Hessians}

In this section we give a brief account of Artinian Gorenstein algebras and Macaulay-Matlis duality. 

\begin{defin}\rm Let $R = \C[x_0,\ldots,x_n] $ be a polynomial ring with the usual grading and $I \subset R$ be a homogeneous Artinian ideal and suppose, without loss of generality that $I_1=0$. 
Then the graded Artinian $\C$-algebra $A=R/I = \displaystyle\bigoplus_{i=0}^dA_i$ is  standard, i.e. it is generated in degree $1$ as an algebra. Since $A$ is Artinian, under the hypothesis $I_1=0$, we call $n+1$ the codimension of $A$, by abuse of notation. Setting $h_i(A)=\dim_\C A_i$, the \emph{Hilbert vector} of $A$ is $\Hilb(A)=(1,h_1(A),\dots,h_d(A))$. The Hilbert vector is sometimes conveniently expressed as the \emph{Hilbert function} of $A$, given by the formula
\begin{equation}
\label{HF}
 H(A,t)=\sum_{k=0}^dh_k(A)t^k.
\end{equation}

\end{defin}

The Hilbert vector $\Hilb(A)$ is said to be \emph{unimodal} if there exists an integer $t\ge 1$ such that $1\le h_1(A) \le \dots\le h_t(A)\ge h_{t+1}(A) \ge\dots\ge h_d(A).$ Moreover the Hilbert vector $\Hilb(A)=(1,h_1(A),\dots,h_d(A))$ is said to be \emph{symmetric} if $h_{d-i}(A)=h_i(A)$ for every $i=0,1,\dots,\lfloor\frac{d}{2}\rfloor$. The next Definition is based in a well known equivalence that can be found in \cite[Prop. 2.1]{MW}.
\begin{defin}\rm
A standard graded Artinian algebra $A$ as above is Gorenstein if and only if $h_d(A)= 1$ and the restriction of the multiplication of the algebra in complementary degree, that is $A_k \times A_{d-k} \to A_d$, is a perfect paring for $k =0,1,\ldots,d$, see \cite{MW}. If $A_j=0$ for $j >d$, then $d$ is called the \emph{socle degree} of $A$.
\end{defin}
It follows that the Hilbert vector $\Hilb(A)$ of a graded Artinian Gorenstein $\C-$algebra $A$ is symmetric. The converse is not true, and $\Hilb(A)$ is not always unimodal for $A$ Artinian Gorenstein.

\begin{ex}\rm     
The first example of a non unimodal Hilbert vector $\Hilb(A)$ of a Gorenstein algebra $A$ was given by Stanley in \cite{St}, namely 
$$(1,13,12,13,1).$$
This algebra $A$ has codimension $13$ and socle degree $4$. 
In \cite{BI} we can find the first known example of a non unimodal Gorenstein Hilbert function in codimension $5$, namely
 $$(1,5, 12 , 22 , 35 , 51 , 70, 91 , 90 , 91 , 70 , 51 , 35 , 22 , 12 , 5 , 1).$$
All Gorenstein $h$-vectors are unimodal in codimension $\leq 3$, see \cite{St}. To the best of the authors knowledge,  it is not known if there is a non unimodal Hilbert vector of a Gorenstein algebra in codimension $4$, see \cite{MN1}. 
\end{ex}

Since our approach is algebro-geometric-differential, we recall a differentiable version of the Macaulay-Matlis duality which is equivalent to polarity in characteristic zero.  We denote by $R_d=\C[x_0,\ldots,x_n]_d$ the $\C-$vector space of homogeneous polynomials of degree $d$. We denote by $Q=\C[X_0,\ldots,X_n]$ the ring of differential operators of $R$, where $X_i := \frac{\partial}{\partial x_i}$ for $i=0,\ldots,n.$ We denote by $Q_k=\C[X_0,\ldots,X_n]_k$ the $\C-$vector space of homogeneous differential operators of $R$ of degree $k$.\\ For each 
integer $k$, with $d\geq k\geq 0$ there exist natural $\C-$bilinear maps $R_d\times Q_k \to R_{d-k}$ defined by differentiation: $$(f,\alpha) \to f_\alpha := \alpha(f).$$
Let $f\in R$ be a homogeneous polynomial of degree $\deg f=d\geq 1$, we define  the \emph{annihilator ideal of $f$} by
$$\Ann (f) :=\left\{\alpha\in Q | \alpha(f)=0\right\}\subset Q.$$ 
 Note that $\Ann(f)_1 \ne 0$ if and only if $X = V(f) \subset \P^n$ is a cone, that is, up-to a linear change of coordinates, the polynomial $f$ depends only of $x_1, \ldots,x_n$. {\it We assume from now on that $V(f)$ is not a cone, and hence that $\Ann(f)_1= 0.$}
Since $\Ann(f)$ is a homogeneous ideal of $Q$, we can define $$A(f)=\frac{Q}{\Ann(f)}.$$
 Then $A(f)$ is the standard graded Artinian Gorenstein $\C$-algebra associated to $f$, given by the Macaulay-Matlis duality, and it satisfies
 $$\begin{cases} A(f)_j=0 \mbox{ for } j >d\\ A(f)_d=\C \end{cases}.$$ A proof of this result can be found in \cite[Theorem 2.1]{MW}.

\begin{ex} \label{exFer}\rm
Take  $f_F=x_0^d+...+x_n^d$, the Fermat type polynomial of degree $d$. In this case the ideal $\Ann(f)$ is generated by $X_iX_j$ for $0 \leq i <j \leq n$
and by $X_0^d-X_j^d$ for $j=1,2,...,n$. The graded part of degree $k$ of $A$ is $A_k = \langle X_0^k,X_1^k,\ldots,X_n^k\rangle$ for $k=1,\ldots,d-1$, and $A_j=\langle x_0^j\rangle $ for $j=0$ and $j=d$. This  determines the Hilbert vector $$\Hilb(A(f_F))=(1,n+1, \ldots,n+1,1).$$
\end{ex}

\begin{defin}
Let $A=\displaystyle{\oplus_{i=0}^d}A_i$ be an Artinian graded $\C-$algebra with $A_d\neq 0$. 
\begin{enumerate}
\item The algebra $A$ is said to have the Weak Lefschetz property, briefly WLP, if there exists an element $L\in A_1$ such that the multiplication map $\bullet L: A_i\to A_{i+1}$ is of maximal rank for $0\leq i \leq d-1$.
\item The algebra $A$ is said to have the Strong Lefschetz property, briefly SLP, if there exists an element $L\in A_1$ such that the multiplication map $L^k: A_i\to A_{i+k}$ is of maximal rank for $0\leq i\leq d$ and $0\leq k\leq d-i$.
\item  We say that $A$ has the Strong Lefschetz property in the narrow sense, if there is $L \in A_1$ such that the linear map $\bullet L^{d-2k}: A_k \to A_{d-k}$ is an isomorphism for all $k\leq d/2$. 
\end{enumerate}
\end{defin}

\begin{rmk}
In the case of standard graded Artinian Gorenstein algebra the two conditions SLP and SLP in the narrow sense are equivalent. 
\end{rmk}
\begin{ex}\label{ex:monomial}\rm This example is due to Stanley  \cite{St} and  Watanabe  \cite{Wa3}. It is considered to be the starting point of the research area of Lefschetz properties for graded algebras. Nowadays there are lots of different proofs for it. 
Consider the graded Artinian Gorenstein algebra
 $$A = \frac{\C[X_0,\ldots,X_n]}{(X_0^{a_0},\ldots,X_n^{a_n})} = \frac{\C[X_0]}{(X_0^{a_0})}\otimes \ldots \otimes \frac{\C[X_n]}{(X_n^{a_n})},$$
  with integers $a_i>0$ for all $i=0,\ldots,n$. 
 It is a monomial complete intersection. Since the cohomology of the complex projective space is $H^*(\P^m,\C)=\C[x]/(x^{m+1})$, and the Segre product commutes with the tensor product by K\"unneth Theorem for cohomology, we have:
 $$H^*(\p^{a_0-1}\times \ldots \times \p^{a_n-1},\C)=\frac{\C[X_0,\ldots,X_n]}{(X_0^{a_0},\ldots,X_n^{a_n})}. $$
 By the Hard Lefschetz Theorem applied to the smooth projective variety $(\p^{a_0-1}\times \ldots \times \p^{a_n-1}$, we know that $A$ has the SLP. 
\end{ex}
The standard graded Artinian Gorenstein algebra $A(f)$ associated to a form $f$ is a natural model for the cohomology algebras of spaces in several categories. For smooth projective varieties, the Hard Lefschetz theorem inspired what is now called Lefschetz properties for 
the algebra $A(f)$. As we show below, the geometric properties of the higher order objects introduced in this paper are intrinsicly connected with such  Lefschetz properties, see Theorem \ref{thm:polarrankhess}.

\subsection{Hessians and  Lefschetz properties}

We recall the following classical results involving the usual Hessian. 
Cones are trivial forms with vanishing Hessian and are characterized by the fact that $Z_X\subset H = \P^{n-1} \subset \P^n$ is a degenerate variety. Hesse claimed in \cite{He} that a reduced hypersurface has vanishing Hessian if and only if it is a cone. Gordan-Noether proved that the claim is true for $n \leq 3$
and false for $n \geq 4$ and this is part of the so called Gordan-Noether theory, see \cite{GN, CRS,Wa3, GR, Go, Ru}.
More precisely, let $f \in R_d$ be a reduced form and let $X = V(f) \subset \P^n $ be the associated hypersurface. Consider the polar map associated to $f$:  
$$
\varphi_X:\p^n\dasharrow(\p^n)^{*}.$$
It is also called the gradient map of $X=V(f)\subset \p^n$, and it is defined by
$$\varphi_X(p)=(f_{x_0}(p):\cdots :f_{x_n}(p)),$$
where $f_{x_i}= \frac{\partial f}{\partial x_i}$.
The image $Z=Z_X$ of $\p^n$ under the polar map $\varphi_X$ is called the polar image of $X$. 
\begin{prop} \label{prop:GNcriteria} \cite{GN} Let $f\in \C[x_0,\ldots,x_n]$ be a reduced polynomial and consider $X  = V(f) \subset \P^n$. Then

\begin{enumerate}
 \item[(i)] $X$ is a cone if and only if $Z\subset H = \p^{n-1}$ is degenerated, which is equivalent to say that  $ f_{x_0},\ldots,f_{x_n} $ are linearly dependent; 
 \item[(ii)] $\hess_f=0$ if and only if $Z \subsetneq \p^n$, or equivalently $f_{x_0},\ldots,f_{x_n}$ are algebraically dependent.
\end{enumerate}

\end{prop}

\begin{thm}\label{thm:GN} \cite{GN} Let $X = V(f) \subset \P^n$, $n \leq 3$, be a hypersurface such that $\hess_f=0$. Then $X$ is a cone.
\end{thm}

\begin{thm}\label{thm:GN2} \cite{GN} For each $n \geq 4$ and $d \geq 3$ there exist irreducible hypersurfaces $X = V(f) \subset \P^n$, 
of degree $\deg(f) = d$, not cones, such that $\hess_f=0$.  
\end{thm}

Now we recall a generalization of a construction that can be found in \cite{MW}. Set $A=A(f)$,  let $k\le l$ be two integers, take $L\in A_1$ and let us consider the linear map
    $$\bullet L^{l-k}: A_k\to A_l.$$
Let $\ba_k=(\alpha_1,\ldots,\alpha_r)$ be a basis of the vector space $A_k$, and 
$\ba_l=(\beta_1,\ldots,\beta_s)$ be a basis of the vector space $A_l.$

\begin{defin}
\label{defMH}
We call mixed Hessian of $f$ of mixed order $(k,l)$ with respect to the basis $\ba_k$ and $\ba_l$ the matrix: 
  $$\Hess_f^{(k,l)}:=[ \alpha_i\beta_j(f)]$$
Moreover, we define $\Hess_f^k=\Hess_f^{(k,k)}$, $\hess_f^k = \det(\Hess_f^k)$ and $\hess_f=\hess_f^1$.
  \end{defin}
Note that $A(f)_1=Q_1$ by our assumption, which implies that $\Hess_f=\Hess_f^1$ is the usual Hessian 
matrix of the polynomial $f$ and $\hess_f$ is the usual  Hessian of $f$. 
Since $A$ is Gorenstein, there is an isomorphism $A_k^* \simeq A_{d-k} $. Therefore, given the basis
$\ba_k=(\alpha_1,\ldots,\alpha_r)$ of $A_k$ and a basis $\theta$ of $A_d \simeq \C$, we get the dual basis
 $\ba^*_k=(\beta^*_1,\ldots,\beta^*_s),$ of $A_{d-k}$ in the following way
  $$\beta^*_i\beta_j(f)=\delta_{ij}\theta.$$

\begin{defin}
We call dual mixed Hessian matrix the matrix 
  $$\Hess^{(k^*,l)}(f):=[(\beta_i^*)\alpha_j(f)]$$
\end{defin}

Note that $\rk \Hess^{(k^*,l)} = \rk \Hess^{(d-k,l)}$.

If $L=a_0X_0+\ldots+a_nX_n \in Q_1,$ 
we set $L^{\perp}=(a_0,\ldots,a_n) \in \C^{n+1}.$ 

The next result can be found in \cite{GZ2} and it is a generalization of the main result of \cite{MW}.

\begin{thm}\cite{GZ2} \label{thm:generalization}
With the previous notation, let $M$ be the matrix associated to the map $\bullet L^{l-k}:A_k \to A_l$ with respect 
to the bases $\ba_k$ and $\ba_l.$ Then
    $$M=(l-k)!\Hess^{(l^*,k)}(f)(L^{\perp}).$$
\end{thm}

\begin{cor} \label{cor1}
For a generic $L$, one has the following.
\begin{enumerate}
\item The map $\bullet L^{d}:A_0 \to A_d$ is an isomorphism.
\item The map $\bullet L^{d-2}:A_1 \to A_{d-1}$ is an isomorphism if and only if $\hess_f \ne 0$.
\item If $d=2k$ is even, then $\hess_f^k \neq 0$.
\item $A$ has the SLP if and only if $\hess^k_f \neq 0$ for all $k\leq d/2$.
\end{enumerate}
\end{cor}

Using Theorem \ref{thm:GN} and Theorem \ref{thm:generalization} and we get the following. 
\begin{cor}\label{cor:lowdeg}
 All standard graded Artinian algebras $A$ of $\codim A \leq 4$ and of socle degree $=3,4$ have the SLP. 
\end{cor}

\begin{cor}\cite{Go} For each pair $(n,d)\not\in\{(3,3), (3,4)\}$ with  $N \geq 3$ and with $d \geq 3$, there exist standard graded 
Artinian Gorenstein algebras $A = \displaystyle \oplus_{i=0}^d A_i$ of codimension $\dim A_1=n+1 \geq 4$ and socle degree $d$ that do not satisfy the Strong 
Lefschetz Property. Furthermore, for each $L \in A_1$ we can choose arbitrarily the level $k$ where the map 
$$\bullet L^{d-2k} : A_k \to A_{d-k}$$ 
is not an isomorphism.
\end{cor}

\begin{rmk}\rm For algebras of codimension $2$, SLP hold in general. Therefore, for Gorenstein algebras, it means that the higher Hessians are not zero. 
This result is a first step in order to generalize Theorem \ref{thm:GN}.
 The issue is that in codimension $3$ the problem is open, that is, we do not know if there is an AG algebra failing SLP.  
 A generalizaion of Theorem \ref{thm:GN2} can be found in \cite{Go}. In this work we give a generalization of Proposition \ref{prop:GNcriteria}, see Theorem \ref{thm:polarrankhess}.
\end{rmk}

\subsection{Jacobian ideals and Milnor algebras}

Let $R=\C[x_0,\ldots,x_n]$ be the polynomial ring in $n+1$ variables with complex coefficients, endowed with the usual grading.\\ Let $f\in R_d$ be a homogeneous polynomial of degree $d$ such that the hypersurface $X=V(f)\subset \p^n$ is reduced. Let $J(f)$ be the Jacobian ideal of $f$, generated by the partial derivatives $f_{x_i}$, of $f$ with respect to $x_i$ for $i=0,\ldots,n$. If $X$ is smooth, then the ideal $J(f)$ is generated by a regular sequence, and $M(f)=R/J(f)$ is a Gorenstein Artinian algebra. Moreover we have $$\dim_{\C} M(f)<+\infty\Leftrightarrow V(f) \mbox{ is a smooth, }$$ 
and the corresponding Hilbert function is given by
\begin{equation} \label{eq0}
H(M(f);t)=\left( \frac{1-t^{d-1}}{1-t}\right)^{n+1}.
\end{equation}
In particular, the socle degree of $M(f) $ is $(d-2)(n+1)$.\\

Assume now that $X\subset \p^n$ is singular, but reduced. In this case the Jacobian algebra is not of finite length, in particular it is not Artinian. It contains information on the structure of the singularities and on the global geometry of $X$.The following results can be found in \cite{IG}.

\begin{prop} \label{propIG1}
 Let $V:f = 0$ be a hypersurface in $\p^n$ of degree $d>2$, such that its singular locus $V_s$ has dimension at most $n-3$. Then $M(f)$ has the WLP in degree $d-2$.
\end{prop}

\begin{prop} \label{propIG2}

Let $V:f = 0$ be a hypersurface in $\p^n$ of degree $d>2$, such that its singular locus $V_s$ has dimension at most $n-3$. Then for every positive integer $k<d-1$ $M(f)$ has the SLP in degree $d-k-1$ at range $k$.
\end{prop} 

\begin{thm} \label{thmIG1}
 Let $V:f = 0$ be a general hypersurface, then $M(f)$ has the SLP.  
\end{thm}
In view of the above result, it is natural to ask the following.

\begin{question}\label{q1}
\rm
Is it true for any homogeneous polynomial $f$ with $V(f)$ smooth?  
\end{question}

We have the following results in relation with this question.

\begin{prop} \label{propSLP1}
 Let $V:f = 0$ be any smooth surface in $\p^3$ of degree $d=3$. Then $M(f)$ has the SLP.
\end{prop}

\begin{proof}
Since $M(f)$ is Artinian Gorenstein, by \cite[Theorem 2.1]{MW}  we have
 $$M(f) \cong Q/\Ann(g),$$
for some homogeneous polynomial $g$,
where  $$\deg(g)= \text{ socle degree of } M(f)= (n+1)(d-2)=4$$ and  $\hess_g\neq 0,$
by Theorem \ref{thm:GN}.  Indeed, otherwise $V(g)$ would be a cone, in contradiction with $\dim M(f)_1=4$. By Corollaries \ref{cor1} and \ref{cor:lowdeg}, $M(f)$ has the SLP.
\end{proof}

\begin{prop} \label{propSLP2}
 Let $V:f = 0$ be a smooth curve in $\p^2$ of even degree $d=2d'$. Then the multiplication by the square of a generic linear form $\ell \in R_1$ induces an isomorphism
 $$\ell^2: M(f)_{3d'-4} \to M(f)_{3d'-2}.$$
In particular, when $d=4$, the Milnor algebra $M(f)$ has the SLP.
\end{prop}

\begin{proof} Note that the socle degree of $M(f)$ is in this case $T=3(d-2)=6d'-6$.
As explained in \cite[Remark 3.7]{DStJump}, a linear form $\ell$ such that the above map is not an isomorphism corresponds exactly to the fact that the associated line $L: \ell=0$ in $\p^2$ is a jumping line of the second kind for the rank two vector bundle $T\langle V \rangle $ on $\p^2$, where $T\langle V \rangle $ is the sheaf of logarithmic vector fields along $V$ as considered for instance in \cite{AD, DS14,MaVa}. Then a key result \cite[Theorem 3.2.2]{KH} of K. Hulek implies that the set of jumping lines of second kind is a curve in the dual projective plane $(\p^2)^{*}$ of all lines.
When $d=4$, this yields an isomorphism $\ell^2: M(f)_{2} \to M(f)_{4}.$
The other isomorphisms necessary for the SLP follows from Corollary \ref{cor1}. 
\end{proof}

\begin{rmk} \label{RkCurves} \rm 
For any smooth curve $V:f=0$ in $\P^2$, the associated Milnor algebra $M(f)$ has the WLP, as follows from the more general results in \cite{HMNW}. In addition, for a singular, reduced curve $V:f=0$ in $\P^2$, the associated Milnor algebra $M(f)$ is no longer Artinian or Gorenstein, but a partial WLP still holds, see \cite[Corollary 4.4]{DP}.

\end{rmk}\rm

\section{Higher order Jacobian ideals and Milnor algebras}

Let us consider the $k$-th order Jacobian ideal of $f \in R$ to be 
$J^k =J^k(f)= (Q_k \ast  f) = (A_k \ast f$), the ideal generated by the $k$-th order partial derivatives of $f$. Then $J^k$ is a homogeneous ideal and we define $M^k=M^k(f)=R/J^k$ to be the $k$-th order Milnor algebra of $f$. For $k=1$, the ideal $J^1$ is just the usual Jacobian ideal $J(f)$ of $f$
and $M^1$ is the usual Milnor algebra $M(f)$ as defined in the previous section.

\begin{rmk} \label{R2} \rm 
For $k=2$, the ideal $J^2(f)$ is the ideal in $R$ spanned by all the second order partial derivatives of $f$. Euler formula implies that $J(f) \subset J^2(f)$, when $d=\deg(f) \geq 2$.
It follows that  $M^2(f)$ coincides with the graded {\it first Hessian algebra} $H_1(f)$ of the polynomial $f$, as defined in \cite{DSt1}.
It follows from \cite[Theorem 1.1]{DSt1}  and \cite[Example 2.7]{DSt1} that, for a hypersurface $V(f)$ having at most isolated singularities, the algebra $M^2(f)=H_1(f)$ is Artinian 
if and only if the multiplicity of the hypersurface $V(f)$ at any singular point is 2.

\end{rmk}\rm 

The above remark can be extended to higher order Milnor algebras. First consider an isolated hypersurface singularity $(V,0):g=0$ at the origin of $\C^n$. Then we define the $k$-th order Tjurina ideal $TI^k(g)$ to by the ideal in the local ring $\OO_n$ generated by all the partial derivatives $\partial^{\alpha}g$, for $0 \leq |\alpha|\leq k$. The $k$-th order Tjurina algebra of the germ $(V,0)$ is by definition the quotient
$$T^k(V,0)=\frac{\OO_n}{TI^k(g)}.$$
It can be shown that this algebra depends only on the isomorphism class of the germ $(V,0)$, and we define the $k$-th Tjurina number of $(V,0)$ to be the integer
$$\tau^k(V,0)=\dim_{\C}T^k(V,0).$$
With this notation, we have the following result.
\begin{thm} \label{T1}
The $k$-th order Milnor algebra $M^k(f)$ of a reduced homogeneous polynomial $f$ is Artinian if and only if the multiplicity of the projective hypersurface $V(f)$ at any point $p \in V(f)$ is at most $k$. Moreover, if the hypersurface $V(f)$ has only isolated singularities, say at the points $p_1,\ldots,p_s$, then for any $k$ and for any large enough $m$ one has
$$\dim_{\C} M^k(f)_m=\sum_{i=1}^s\tau^k(V,p_i).$$
\end{thm}

\proof The algebra $M^k(f)$ is Artinian if and only if the zero set $Z(J^k(f))$ of the ideal $J^k(f)$ in $\p^n$ is empty. Note that a point $p \in Z(J^k(f))$ is a point on the hypersurface $V(f)$, by a repeated application of the Euler formula. If we choose the coordinates on $\p^n$ such that $p=(1:0:\ldots:0)$, then the local equation of the hypersurface germ $(V(f),p)$ is
$$g(y_1,\ldots,y_n)=f(1,y_1,\ldots,y_n)=0,$$
exactly as in the proof of \cite[Theorem 1.1]{DSt1}. It follows that the localization of the ideal $J^k(f)$ at the point $p$ coincides with the ideal generated by all the partial derivatives $\partial^{\alpha}g$, for $0 \leq |\alpha|\leq k$. And these derivatives vanish all at $p$ exactly when the multiplicity of $V(f)$ at $p$ is $>k$. This proves the first claim.
The proof of the second claim is completely similar.
\endproof

Note that \cite[Example 2.18]{DSt1} shows that even for a smooth curve $V(f)$, the algebra $M^2(f)=H_1(f)$ is not Gorenstein in general, since its Hilbert function, which depends on the choice of the smooth curve, is not symmetric and the dimension of the socle can be $>1$.
However, there is a Zariski open subset $U_{d,k}$ in $R_d$ such that the Hilbert vector
$\Hilb(M^k(f))$ is constant for $f \in U_{d,k}$. 
\begin{question}\label{q2}
\rm
Determine the value of the vector $\Hilb(M^k(f))$, or equivalently of the Hilbert function $H(M^k(f),t)$ for $f \in U_{d,k}$. 
\end{question}
By semicontinuity, it follows that
$$h_i(M^k(f))= \min \{ \dim(M^k(g)_i) \ : \ g \in R_d\}.$$

Similarly, there is a Zariski open subset $U'_{d}$ in $R_d$ such that the Hilbert vector
$\Hilb(A(f))$ is constant for $f \in U'_{d}$. 
Using recent results by Zhenjian Wang, see \cite[Proposition 1.3]{ZW2}, we have the following.

\begin{prop} \label{propZW}
 For a  polynomial $f \in U'_d$, one has $h_k(A(f))=\dim Q_k={n+k \choose n}$ for
 $k \leq d/2$ and $h_k(A(f))=h_{d-k}(A(f))={n+d-k \choose n}$ for $d/2<k\leq d$.
 In particular, a Fermat type polynomial $f_F=x_0^d+ \ldots +x_n^d$ is not in $U'_d$, for $d \geq 4$.
\end{prop}
\proof
Since $A(f)$ is Artinian Gorenstein with socle degree $d$, it is enough to prove only the claim for $k \leq d/2$. This claim is equivalent to $\Ann(f)_k=0$ for $k \leq d/2$, and also to
$\dim J^k(f)_{d-k}=\dim Q_k$. This last equality is exactly the claim of \cite[Proposition 1.3]{ZW2},
where $J^k(f)_{d-k}$ is denoted by $E_k(f)$. The claim for the Fermat type polynomial follows from Example \ref{exFer}.
\endproof

\begin{cor} \label{cor2}
For any $k \leq d/2$ and any polynomial $f \in U_{d,k}$ one has  $$h_i(M^k(f))= {n+i \choose n}\text{ for } i<d-k$$ 
and
$$h_{d-k}(M^k(f))=\dim R_{d-k}-\dim Q_k={n+d-k \choose n}-{n+k \choose n}.$$
In particular, a Fermat type polynomial $f=x_0^d+ \ldots +x_n^d$ is not in $U_{d,k}$, for $d \geq 2k \geq 4$.
\end{cor}

In conclusion, the introduction of higher order Milnor algebras is motivated by the desire to construct a larger class of Artin graded algebras starting with homogeneous polynomials. These new Artinian algebras may exhibit interesting examples with respect to Lefschetz properties. It is known that the Hessians are related to Lefschetz properties, and the Hessians of singular hypersurfaces behave in a different way from the ones of smooth hypersurfaces. As an example, we have the following. Recall first that $\hess_f$ denotes the Hessian of
a homogeneous polynomial $f$ as in Definition \ref{defMH} or, more explicitly,
$$\hess_f= \det \left( \frac{\partial ^2 f}{\partial x_i \partial x_j}\right)_{0\leq i,j \leq n}.$$
\begin{prop} \label{P1.5}
Let $f$ be a homogeneous polynomial in $R$.
\begin{enumerate}

\item If the hypersurface $V(f)$ is smooth, then $\hess_f \notin J(f)$.

\item If the hypersurface $V(f)$ is not smooth, but has isolated singularities, then $\hess_f \in J(f)$.
\end{enumerate}

\end{prop}

\proof The first claim is well known, and it holds in fact for any isolated hypersurface singularity, not only for the cone over $V(f)$, see Theorem 1, section 5.11 in \cite{AGV}.
The second claim is less known, and it follows from \cite[Proposition 1.4 (ii)]{DSt1}. Indeed, for the $n+1$-st Hessian algebra $H_{n+1}(f)$, one has the equalities
$$H_{n+1}(f)=\frac{R}{J(f)+(\hess_f)}=\frac{M(f)}{({\overline \hess_f})},$$
where $(\hess_f)$ is  the principal ideal in $R$ generated by the Hessian $\hess_f$ and $({\overline \hess_f})$ is  the principal ideal in $M(f)$ generated by the class $ {\overline \hess_f}$ of the
Hessian $\hess_f$ in $M(f)$.
 By \cite[Proposition 1.4 (ii)]{DSt1}, we know that the graded algebras $H_{n+1}(f)$ and $M(f)$ have the same Hilbert series when the hypersurface $V(f)$ is not smooth and has only isolated singularities. This proves our claim (2).
\endproof

\begin{question} \label{questionHESS}
Is is true that $\hess_f \in J(f)$ for any reduced, singular hypersurface $V(f)$?
\end{question}

\section{The Hilbert functions of $A(f)$ and $M^k(f)$, and the geometry of $V(f)$}

It is known that for two homogeneous polynomials $f, g \in R_d$, the corresponding Milnor algebras $M(f)$ and $M(g)$ are isomorphic as $\C$-algebras if and only if the associated hypersurfaces $V(f)$ and $V(g)$ in $\p^n$ are projectively equivalent. This claim follows from \cite{MY} when the hypersurfaces $V(f)$ and $V(f')$ are both smooth. However, the method of proof can be extended to cover all hypersurfaces. 
For a closely related result, see \cite{ZW1}.

Note that a similar claim fails if we replace the Milnor algebra $M(f)=M^1(f)$ by the second order Milnor algebra $M^2(f)$. Indeed, it is enough to consider the family of complex plane cubics $f_a=x_0^3+x_1^3+x_2^3-3ax_0x_1x_2$, where $a\ne0$, $a^3\ne 1$.
In this case $M^2({f_a})=R/(x_0,x_1,x_2)=\C$ does not detect the parameter $a$. However, the main result of \cite{ZW2} implies the following.

\begin{thm} \label{thmZW2}
The $k$-th order Milnor algebra $M^k(f)$ of a generic homogeneous polynomial $f$ determines the hypersurface $V(f)$ up-to projective equivalence, when $k \leq d/2-1$.
\end{thm}

\proof
Let $f$ and $g$ be two generic, degree $d$ homogeneous polynomials in $R$, such that we have an isomorphism of graded algebras
$$M^k(f)=R/J^k(f) \simeq R/J^k(g)=M^k(g).$$
Then there is a linear change of coordinates $\phi \in Gl_{n+1}(\C)$, inducing the above isomorphism, and hence such that 
$\phi^*(J^k(f))=J^k(g)$. This last equality can be re-written as $J^k(f\circ \phi)=J^k(g)$, and Theorem 1.2 and Proposition 1.3 in  \cite{ZW2} imply that the two hypersurfaces $V(f\circ \phi)$ and $V(g)$ coincide.
\endproof

Saying that the associated hypersurfaces $V(f)$ and $V(g)$ in $\p^n$ are projectively equivalent means that the two polynomials $f$ and $g$ belongs to the same $G$-orbit, where $G=Gl_{n+1}(\C)$ is acting in the natural way on the space of polynomials $R_d$ by substitution. And the fact that $f$ and $g$ belongs to the same $G$-orbit implies immediately that the Milnor algebras $M(f)$ and $M(g)$ are isomorphic.
Similarly,  the fact that $f$ and $g$ belong to the same $G$-orbit implies immediately that the standard graded Artinian Gorenstein algebras $A(f)$ and $A(g)$ are isomorphic, see for instance \cite[Lemma 3.3]{DiPo}. In particular, the Hilbert function of $A(f)$ is determined by the $G$-orbit of $f$, and hopefully by the geometry of the corresponding hypersurface $V(f)$. However, the following example seems to suggest that it is a hard problem to relate the geometry of the hypersurface $V(f)$ to the properties of the algebras $A(f)$, $M(f)$ and $M^2(f)$.

\begin{ex} \label{ex3}\rm
In this example we look at quartic curves in $\p^2$, i.e. $(n,d)=(2,4)$.
When $V(f)$ is not a cone, only dimension $h_2(A(f))$ has to be determined. It turns out that all the possible values $ \{3,4,5,6\}$ are obtained. All the computations below were done using CoCoA software, see \cite{CoCoA}, with the help of Gabriel Sticlaru.

\medskip

\noindent{\bf Case $V(f)$ smooth}

\medskip

All the smooth quartics $V(f)$ have the same Hilbert function
$$H(M(f);t)=1+3t+6t^2+7t^3+6t^4+3t^5+t^6,$$
given by the formula \eqref{eq0}. But the other invariants may change, as the following examples show. 

\begin{enumerate}

\item When $f_F=x_0^4+x_1^4+x_2^4$, the Fermat type polynomial of degree $4$, the corresponding Hilbert function is 
$$H(A(f_F);t)=1+3t+3t^2+3t^3+t^4,$$
by Example \ref{exFer}. The minimal resolution of $A(f_F)$
is given by
$$0 \to  Q(-7) \to  Q(-3)^2 \oplus Q(-5)^3 \to  Q(-2)^3 \oplus Q(-4)^2 \to Q,$$
in particular $A(f_F)$ is not a complete intersection. The second order Milnor algebra is $M^2({f_F})=R/(x_0^2,x_1^2,x_2^2)$, hence a complete intersection, with Hilbert function
$$H(M^2({f_F});t)=1+3t+3t^2+t^3.$$

\item For the smooth Caporali quartic given by $f_{Ca}=x_0^4+x_1^4+x_2^4+(x_0+x_1+x_2)^4$, we get
$$H(A(f_{Ca});t)=1+3t+4t^2+3t^3+t^4,$$
and the minimal resolution of $A(f_{Ca})$
is given by
$$0 \to  Q(-7) \to  Q(-4) \oplus Q(-5)^2 \to  Q(-2)^2 \oplus Q(-3) \to Q.$$
Hence $A(f_{Ca})$ is a complete intersection of multi-degree $(2,2,3)$.
The second order Milnor algebra  $M^2({f_{Ca}})$ has a Hilbert function given by
$$H(M^2({f_{Ca}});t)=1+3t+2t^2,$$
in particular this algebra is not Gorenstein.

\item For the smooth quartic given by $f_{Ca_1}=x_0^4+x_1^4+x_2^4+(x_0^2+x_1^2+x_2^2)^2$, we get
$$H(Af_{Ca_1});t)=1+3t+6t^2+3t^3+t^4,$$
which coincides with the generic value given in Proposition \ref{propZW}, and the minimal resolution of $A(f_{Ca_1})$
is given by
$$0 \to  Q(-7) \to   Q(-4)^7 \to  Q(-3)^7  \to Q.$$
Hence $A(f_{Ca_1})$ is far from being a complete intersection.
The second order Milnor algebra $M^2({f_{Ca_1}})$ has a Hilbert function given by
$$H(M^2({f_{Ca_1}});t)=1+3t,$$
in particular this algebra is not Gorenstein.

\item For the smooth quartic given by $f_{Ca_2}=x_0^4+x_1^4+x_2^4+(x_0^2+x_1^2)^2$, we get the same Hilbert function as for $A(f_{Ca})$,
but the minimal resolution of $A(f_{Ca_2})$
is given by
$$0 \to  Q(-7) \to Q(-3)\oplus  Q(-4)^2 \oplus Q(-5)^2 \to  Q(-2)^2 \oplus Q(-3)^2\oplus Q(-4)  \to Q.$$
The second order Milnor algebra  $M^2({f_{Ca_2}})$ has a Hilbert function given by
$$H(M^2({f_{Ca}});t)=1+3t+2t^2,$$
and hence this algebra is again not Gorenstein.
\end{enumerate}

\medskip

\noindent{\bf Case $V(f)$ singular}

\medskip

\begin{enumerate}

\item The rational quartic with an $E_6$-singularity, defined by $f_C=x_0^3x_1+x_2^4$ satisfies
$H(A(f_{C});t)=H(A_{f_F};t)$ and the minimal resolution for $A_{f_{C}}$ is
$$0 \to  Q(-7) \to Q(-3)^2 \oplus Q(-5)^3 \to  Q(-2)^3 \oplus Q(-4)^2  \to Q.$$
Hence the algebra $A(f_{C})$ has the same resolution as a graded $R$-module as the algebra $A(f_{F})$. But these two algebras are not isomorphic. Indeed, note that
$$\Ann(f_F)_2=\langle X_0X_1,X_0X_2,X_1X_2\rangle \text{ and } \Ann(f_C)_2= \langle X_1^2,X_0X_2,X_1X_2\rangle.$$
An isomorphism $A(f_F) \simeq A(f_C)$ of $\C$-algebras would imply that the two nets of conics
$$N_F: aX_0X_1+bX_0X_2+c X_1X_2 \text{ and } N_C: aX_1^2+b X_0X_2+c X_1X_2$$
are equivalent. This is not the case, since a conic in $N_F$ is singular if and only if it belongs to the union of  three lines given by $abc=0$, while a conic in $N_C$ is singular if and only if it belongs to the union of two lines given by $ab=0$.
For the associated Milnor algebras, one has
$$H(M(f_C);t)=1+3t+6t^2+7t^3+7t^4+6 \frac{t^5}{1-t},$$
and
$$H(M^2(f_C);t)=1+3t+3t^2+2 \frac{t^3}{1-t}.$$
Hence $M^2(f_C)$ is not Artinian, as predicted by Theorem \ref{T1}.
Note that this curve has a unique $E_6$-singularity, with Tjurina numbers 
$\tau(E_6)=\tau^1(E_6)=6$ and $\tau^2(E_6)=2$, which explain the coefficients of the rational fractions in the above formulas, in view of Theorem \ref{T1}.
\item For $f_{3A_2}=x_0^2x_1^2+x_1^2x_2^2+x_0^2x_2^2-2x_0x_1x_2(x_0+x_1+x_2)$, which defines a quartic curve with 3 cusps $A_2$, a direct computation shows that
$$H(A(f_{3A_2});t)=1+3t+6t^2+3t^3+t^4,$$
which coincides with the generic value given in Proposition \ref{propZW}. The minimal resolution is
$$0 \to  Q(-7) \to   Q(-4)^7 \to  Q(-3)^7  \to Q.$$
Hence the algebra $A(f_{3A_2})$ has the same resolution as a graded $R$-module as the algebra $A(f_{Ca1})$. Does this imply that these two algebras are isomorphic? In this case $\Ann(f)_2=0$ and $\dim \Ann(f)_3=7$, hence it is more complicated to use the above method to distinguish these two algebras. Note also that the line arrangement $f=x_0x_1x_2(x_0+x_1+x_2)=0$ gives rise to an algebra $A(f)$ with exactly the same 
resolution as a graded $R$-module as the algebra $A(f_{Ca1})$.

For the associated Milnor algebras, one has
$$H(M(f_{3A_2});t)=1+3t+6t^2+7t^3+6 \frac{t^4}{1-t},$$
and
$$H(M^2(f_{3A_2});t)=1+3t.$$
Hence $M^2(f_{3A_2})$ is  Artinian, as predicted by Theorem \ref{T1}, but not Gorenstein.

\item For $f_{2A_3}=x_0^2x_1^2+x_2^4$, which defines a quartic curve with $2$ singularities $A_3$, a direct computation shows that
$$H(A(f_{2A_3});t)=1+3t+4t^2+3t^3+t^4$$
and the minimal resolution is
$$0 \to  Q(-7) \to Q(-3)\oplus  Q(-4)^2 \oplus Q(-5)^2 \to  Q(-2)^2 \oplus Q(-3)^2\oplus Q(-4)  \to Q.$$
Hence the algebra $A(f_{2A_3})$ has the same resolution as a graded $R$-module as the algebra $A(f_{Ca_2})$. Does this imply that these two algebras are isomorphic? Note that
$$\Ann(f_{2A_3})_2=\langle X_0X_2,X_1X_2\rangle =\Ann(f_{Ca_2})_2,$$
while $\Ann(f_{2A_3})_3$ and $\Ann(f_{Ca_2})_3$ have both dimension $7$.
Hence again the above method to distinguish these two algebras is not easy to apply. For the associated Milnor algebras, one has
$$H(M(f_{2A_3});t)=H(M(f_C);t)$$
and
$$H(M^2(f_{2A_3});t)=1+3t+2t^2.$$
Hence $M^2(f_C)$ is  Artinian, but not Gorenstein.

\item For $f_{4A_1}=(x_0^2+x_1^2)^2+(x_1^2+x_2^2)^2$, which defines a quartic curve with $4$ singularities $A_1$ that is the union of two conics intersecting in the $4$ nodes, a direct computation shows that
$$H(A(f_{4A_1});t)=1+3t+5t^2+3t^3+t^4$$
and the minimal resolution is
$$0 \to  Q(-7) \to  Q(-4)^4 \oplus Q(-5) \to  Q(-2) \oplus Q(-3)^4  \to Q.$$

\item For the line arrangement defined by $f=(x_0^3+x_1^3)x_2$, we get 
an algebra $A(f)$ with exactly the same resolution as a graded $R$-module as the algebra $A(f_{Ca})$, hence a complete intersection of multi-degree $(2,2,3)$. But these two algebras are not isomorphic. Indeed, note that
$$\Ann(f_{Ca})_2=\langle X_0X_1-X_1X_2,X_0X_2-X_1X_2\rangle \text{ and } \Ann(f)_2=\langle X_0X_1,X_2^2\rangle.$$
An isomorphism $A(f_{Ca}) \simeq A(f)$ of $\C$-algebras would imply that the two pencils of conics
$$P_{Ca}: a(y_0y_1-y_1y_2)+b(y_0y_2-y_1y_2) \text{ and } P_f: ay_0y_1+b y_2^2$$
are equivalent. This is not the case, since a conic in $P_{Ca}$ is singular if and only if it belongs to the union of three lines given by $ab(a+b)=0$, while a conic in $P_f$ is singular if and only if it belongs to the union of two lines given by $ab=0$.

For the associated Milnor algebras, one has
$$H(M(f_{4A_1});t)=1+3t+6t^2+7t^3+6t^4+4 \frac{t^5}{1-t},$$
and
$$H(M^2(f_{4A_1});t)=1+3t+t^2.$$
Hence $M^2(f_{4A_1})$ is  Artinian and Gorenstein.

\end{enumerate}

\end{ex}

\begin{ex} \label{ex4}\rm
We show here that for some smooth quartics $V(f)$ in $\p^2$ the multiplication by a generic linear form $\ell$ does not give rise to an injection $M^2(f)_1 \to M^2(f)_2$. Note that one has
$$\dim M^2(f)_1=\dim R_1=3 \text{ and } \dim M^2(f)_2=\dim R_2-\dim A(f)_{2}= 6-\dim A(f)_{2},$$
using the general formula 
$$\dim J^k_{d-k}=\dim A(f)_{k}.$$
Hence, as soon as $\dim A(f)_{2}\geq 4$, the morphism $M^2(f)_1 \to M^2(f)_2$ cannot be injective. This happens for all the smooth quartics 
in Example \ref{ex3}, except for the Fermat one.
\end{ex}

\begin{question} \label{questionMY} \rm

It would be interesting to find out whether Mather-Yau result extends to this setting, i.e. if an isomorphism
$A(f) \simeq A(f')$ of $\C$-algebras implies that $f$ and $f'$ belongs to the same $G$-orbit. In the case when $(n,d)=(1,4)$ or $(n,d)=(2,3)$, one has $T=(n+1)(d-2)=d$ and
the $G$-orbits of polynomials in $R_d$ corresponding to smooth hypersurfaces can be listed using 1-parameter families. In both cases, the correspondence $\Mac:f_t \to g_{u(t)}$ between a polynomial $f_t\in R_d$ with $V(f_t)$ smooth and a polynomial $g_{u(t)}=\Mac(f_t) \in R_d$, defined by the property that one has an isomorphism
$$M(f_t) =A_{g_{u(t)}},$$
give rise to a bijection $u$ of the corresponding parameter space, see \cite{DiPo} for details. Hence the Mather-Yau result implies a positive answer to our question of the algebras $A_f$ in these two special cases.
\end{question}

\section{Higher Jacobians and Higher Polar mappings}

Consider the following exact sequence 
$$0 \to I_k \to Q_k \to J^k_{d-k} \to 0,$$
where $I=\Ann(f)$ and $J^k=J^k(f)$.
The map $Q_k \to J^k_{d-k}$ is  given by evaluation $ \alpha \mapsto \alpha(f)$. \\
We have also the natural exact sequence:
$$0 \to I_k \to Q_k \to A_k \to 0.$$
Note that the vector spaces $J^k_{d-k}$ and $A_k$ have the same dimension, 
\begin{equation}\label{eq1}
\dim J^k_{d-k} = \dim Q_k -\dim I_k = \binom{n+k}{k} - \dim I_k = \dim A_k.
\end{equation}
\begin{defin}\rm
The $k$-th polar mapping (or $k$-th gradient mapping) of the hypersurface $X=V(f) \subset \P^n$ is the rational map  $\Phi^k_X:  \P^n  \dashrightarrow   \P^{\binom{n+k}{k}-1}$ given by the $k$-th partial derivatives of $f$.  
The $k$-th polar image of $X$ is  $\tilde{Z}_k = \overline{\Phi^k_X(\P^n)}\subseteq\p^{\binom{n+k}{k}-1}$, the closure of the image of the $k$-th polar map.
\end{defin}
For $\{\alpha_1, \ldots, \alpha_{a_k}\}$ a basis of the vector space $A_k$, we define the relative $k$-th polar map of $X$ to be the map $\varphi^k_X:\p^n \dashrightarrow \p^{a_k-1}$ given by the linear system $J^k_{d-k}$:
$$\  \varphi^k_X(p)=(\alpha_1(f)(p): \ldots : \alpha_{a_k}(f)(p) ).$$
The $k$-th relative polar image of $X$ is  $Z_k = \overline{\varphi^k_X(\P^n)}\subseteq \p^{a_k-1}$, the closure of the image of the relative $k$-th polar map.\\

It follows from Proposition \ref{propZW}, that for $k \leq d/2$ and $f$ generic, one has $a_k=\binom{n+k}{k}$ and hence $\Phi^k_X=\varphi^k_X$ in such cases.
In general, the exact sequence $$0 \to I_k \to Q_k \to J^k_{d-k} \to 0$$ gives rise to a linear projection $\P^{\binom{n+k}{k}-1} \dashrightarrow \p^{a_k-1}$ making compatible these two polar maps, as in the diagram
$$\begin{array}{ccc}
   \P^n   & \rightarrow &  \P^{\binom{n+k}{k}-1}\\
          &     \searrow            &   \downarrow \\ 
      &                 &   \p^{a_k-1}.
  \end{array}
$$
Moreover, since $\tilde{Z}_k \subset \P(J^k_{d-k}) = \P^{a_k - 1} \subset \P^{\binom{n+k}{k}-1}$, the secant variety of $\tilde{Z}_k$ does not intersect the projection center, hence $Z_k \simeq \tilde{Z}_k$.
The next result is a formalization of the intuitive idea that the mixed Hessian $\Hess_f^{(1,k)}$, introduced in Definition \ref{defMH}, is the Jacobian matrix of the gradient (or polar) map $\varphi^k$ of order $k$.

\begin{thm}\label{thm:polarrankhess}
 With the above notations, we get:
 $$\dim Z_k = \dim \tilde{Z}_k = \operatorname{rk}(\operatorname{Hess}^{(1,k)}_X)-1.$$
 In particular, the following conditions are equivalents:
 \begin{enumerate}
  \item[(i)] $\varphi^k$ is a degenerated map, that is, $\dim Z_k < n$;
  \item[(ii)] $\rk (\Hess^{(1,k)}_X) < n+1$;
  \item[(iii)] The map $\bullet L^{d-k-1}: A_1 \to A_{d-k}$ has not maximal rank for any $L \in A_1$.
 \end{enumerate}

\end{thm}

\begin{proof}
If $p=[\mathbf v] \in \p^n$, then  $T_p\p^n=\mathbb C^{n+1}/<\mathbf v>$ is  the affine tangent space to $\p^n$ at $p$. Let
$\Hess^{(1,k)}_X(p)$ be the equivalence class of the mixed Hessian matrix of $X=V(f)$ evaluated at $\mathbf v$.
$\Hess^{(1,k)}_X(p)$ passes to the quotients and it induces the differential
of the map $\varphi^k_X$ at $p$:
\begin{equation}\label{eq:dfp}
(d\varphi^k_X)_p: T_p\p^n\to T_{\varphi^k_X(p)}\p^{a_k-1},
\end{equation}
whose image is exactly $T_{\varphi^k_X(p)}Z_k$, when $p$ is generic. From this we can
describe explicitly the projective tangent space to $Z_k$ at $\varphi^k_X(p)$, obtaining 
\begin{equation}\label{eq:TpZ}
T_{\varphi^k_X(p)}Z_k=\p(\Im(\Hess^{(1,k)}_X(v))\subseteq \p^{a_k-1}.
\end{equation}
Thus, there is an integer $\gamma_k \geq 0$ such that
\begin{equation}\label{eq:dimZdiff}
\dim Z_k = \operatorname{rk}(\operatorname{Hess}^{(1,k)}_X)-1=n-\gamma_k.
\end{equation}
The equivalence between $(i)$ and $(ii)$ is clear, since $\rk (\Hess^{(1,k)}_X) < n+1$ if and only if $\gamma_k > 0$ if and only if $\dim (Z_k) < n$. \\
The equivalence between $(ii)$ and $(iii)$ follows from Theorem \ref{thm:generalization}.
\end{proof}

Recall that for a standard graded Artinian Gorenstein  algebra, the injectivity of $\bullet L: A_i \to A_{i+1}$ for a certain $L \in A_1$ implies the injectivity of $\bullet L: A_j \to A_{j+1}$ for all $j < i$, see \cite[Proposition 2.1]{MMN}.
If $A(f)$ does not satisfy the WLP, then there is a minimal $k$ such that $\bullet L: A_k \to A_{k+1}$ is not injective for all $L \in A_1$. By Theorem \ref{thm:polarrankhess} the previous condition is equivalent to 
say that for all $j \leq k$ the matrix  $\Hess_f^{(1,j)}$ has full rank.
The following Corollary is a partial generalization of the Gordan-Noether Hessian criterion, recalled in Proposition \ref{prop:GNcriteria}, to the case of higher order Hessians.

\begin{cor} \label{cor:polar1}
Let $k\leq \lfloor  \frac{d}{2}\rfloor$ be the greatest integer such that $\bullet L: A_{k-1} \to A_k$ is injective for some $L \in A_1$. For each $j \leq k$, we get that $\varphi^j$ degenerated implies $\hess^j_f = 0$. 
\end{cor}

\begin{proof}
 The result follows from the commutative diagram, by Theorem \ref{thm:generalization} and by Theorem \ref{thm:polarrankhess}.
$$\begin{array}{ccc}

A_j & \to & A_{d-j}\\
\uparrow & \nearrow & \\
A_1 &   

  \end{array}
$$
Indeed, for $j\leq k$, we have that $\bullet L^{j-1}:A_1 \to A_j$ is injective, by composition. If $\bullet L^{d-2j}: A_j \to A_{d-j}$ is injective, then $\bullet L^{d-j-1} : A_1 \to A_{d-j}$ is injective. 
In other words, if $\hess^j_f \neq 0$, then $\varphi^j$ is not degenerated.
 
 \end{proof}

The converse is not true, as one can see in the next example.

\begin{ex}\rm Let $f=xu^3+yu^2v+zuv^2+v^4\in \C[x,y,z,u,v]_4$ as in \cite[Example 3]{Go},  and let $A=Q/\Ann(f)$. The map $\bullet L: A_1\to A_2$ is injective for $L=u+v$. For $j=2$, we have $$\Hess^2_f=\begin{pmatrix} 
0&0&0&0&0&0&6&0\\
0&0&0&0&0&2&0&0\\
0&0&0&0&0&0&0&2\\
0&0&0&0&0&0&2&0\\
0&0&0&0&0&2&0&0\\
0&2&0&0&2&0&0&0\\
6&0&0&2&0&0&0&0\\
0&0&2&0&0&0&0&24
\end{pmatrix}$$ and $\hess^2_f=\det\Hess^2_f=0$. Calculating the $\Hess_f^{(1,2)}$, we get: $$\Hess_f^{(1,2)}=\begin{pmatrix}
0&0&0&6u&0\\
0&0&0&2v&2u\\
0&0&0&0&2v\\
0&0&0&2u&0\\
0&0&0&2v&2u\\
0&2u&2v&2y&2z\\
6u&2v&0&6x&2y\\
0&0&2u&2z&24v
\end{pmatrix}.$$   The $\rk(\Hess_f^{(1,2)})=5$, hence $\varphi^2$ is not degenerated, by Theorem \ref{thm:polarrankhess}. 
\end{ex}

\begin{cor} \label{corHess10}
 Let $f$ be a homogeneous form of degree $d$ and let $1<k<d-1$. Let $\varphi^k_f$ be the $k$-th polar map 
of $f$. If $\hess_f \neq 0$ then $\varphi^k_f$ is not degenerated, that is $\dim Z_k = n$.
In particular, if $X = V(f)\subset \P^n$ with $n \leq 3$ is not a cone, then $\varphi^k_X$ is not degenerated.  
\end{cor}

\begin{proof}
 By Theorem \ref{thm:polarrankhess} we have to prove that $rk(\Hess^{(1,k)}_f)=n+1$ is maximal. Let $L\in A_1$ be a general linear form. 
 Since $\hess_f \neq 0$, the multiplication map $\bullet L^{d-2}: A_1 \to A_{d-1}$ is an isomorphism. Indeed, by Theorem \ref{thm:generalization}, after choosing basis the matrix of $\bullet L^{d-2}$
 is $(d-2)!\Hess^{((d-1)^*,1)}$ whose rank is the same as the rank of $\Hess_f$. Since $\bullet L^{d-2}: A_1 \to A_{d-1}$ factors via $\bullet L^{d-k-1}: A_1 \to A_{d-k}$, the injectivity of $\bullet L^{d-2}$ implies 
 the injectivity of $\bullet L^{d-k-1}$. On the other hand the rank of $\Hess^{((d-k)^*,1)}$ is equal to the rank of $\Hess_f^{(1,k)}$. The result now follows from Gordan-Noether theory.
\end{proof}

\begin{ex}\rm Let $X = V(f)\subset \P^3$ be Ikeda's surface given by $f=xuv^3+yu^3v+x^2y^3$ as in \cite{Ik, MW}. Since $X$ is not a cone, by the Gordan-Noether criterion, $\hess_f \neq 0$. On the other hand, $\hess^2_f=0$.
By Corollary \ref{corHess10}, we know that the second polar map is not degenerated. Moreover, $\varphi_X = \Phi_X:\P^3 \dashrightarrow \P^9$ since $\dim A_2 = 10$. 
From an algebraic viewpoint we have $A_1 \to A_2 \to A_3$ and we know that $\bullet L: A_2 \to A_3$ is not an isomorfism, in particular it has non trivial kernel. On the other hand, $\bullet L^2: A_1 \to A_3$ is injective by Theorem \ref{thm:polarrankhess}. 
So the image of the first multiplication does not intersect the kernel of the second one. 
\end{ex}

Consider the $k$-th polar map $\Phi^k_X:\p^n \to \p^N$ associated to a smooth hypersurface $X=V(f)$, where $N={n+k \choose k}-1$.
Corollary \ref{corHess10} implies that the image of this map has dimension $n$ for any $k$ with
$1<k<d-1$. There is a related result: consider the restriction
\begin{equation} \label{eqPol}
\psi^k_X=\Phi^k_X|X:X \to \p^N.
\end{equation}
Then it is shown in \cite{D} that the map $\psi^k_X$ is finite for $1\leq k<d$. The case $k=2$ is stated in \cite[Proposition 1.6]{D} and the general case in  \cite[Remark 2.3 (ii)]{D}. We prove next that the map
$\Phi^k_X$ is finite as well, which implies that $\dim Z_k=n$. Note that this map $\Phi^k_X$ is well defined as soon as $M^k(f)$ is Artinian.

\begin{thm} \label{T2}
Let $\Phi^k_X:\p^n \to \p^N$ be the $k$-th polar map  associated to a hypersurface $X=V(f)\subset \p^n$ such that $M^k(f)$ is Artinian, where $N={n+k \choose k}-1$ and $0<k<d$. Then $\Phi^k_X$ is finite. 
In particular, $$\dim Z_k=\dim \Phi^k_X(\p^n)=n$$
and one has $$\deg \Phi^k_X \cdot \deg \Phi^k_X(\p^n)=(d-k)^n.$$
\end{thm}

\proof Suppose there is a curve $C$ in $\p^n$ such that $\Phi^k_X$ is constant on $C$, say
$$\Phi^k_f(x)=(b_{\alpha})_{|\alpha|=k} \in \p^N,$$
for any $x \in C$. There is at least one multi-index $\beta$ such that $b_{\beta}\ne 0$. But this leads to a contradiction, since either the partial derivative 
$\partial^{\beta}f$ is identically zero, or the hypersurface $\partial^{\beta}f=0$ meets the curve $C$, say at a point $p$. At such a point $p$, all the partial derivatives $\partial^{\alpha}f(p)=0$, for $|\alpha|=k$, in contradiction to the fact that $M^k(f)$ is Artinian.
The claim about the degree follows by cutting
$\Phi^k_X(\p^n)$ with $n$ generic hyperplanes $H_j$ in $\p^N$, where $j=1,...,n$ and using the fact that $(\Phi^k_X)^{-1}(H_j)$ is a family of $n$ hypersurfaces of degree $d-k$ meeting in $\deg \Phi^k_X \cdot \deg \Phi^k_X(\p^n)$ simple points.
\endproof

{\bf Acknowledgments}. The authors would like to thank the organizers and the participants of the workshop {\it Lefschetz Properties and Jordan Type in Algebra, Geometry and Combinatorics}, that took place in Levico Terme, Trento, Italy in June 2018. 
This work was started at this  conference. \\
The second author would like to thank Francesco Russo and Giuseppe Zappalà  for many conversations on this subject.


\begin{thebibliography}{HMMNWW}

\bibitem[AD]{AD}  T. Abe and A. Dimca, {\it Splitting types of bundles of logarithmic vector fields along plane curves}, Internat. J. Math.
29 (2018) 1850055 (20 pages).


\bibitem[AGV]{AGV}  V.I. Arnold, S.M. Gusein-Zade, A.N. Varchenko,  {\it Singularities of Differentiable
Maps}, Vol. 1,  Monographs in Mathematics 82,
Birkhauser, Boston, 1985.

\bibitem[BI]{BI} D. Bernstein and A. Iarrobino, {\it A non-unimodal graded Gorenstein Artin algebra in codimension five}, Comm. Algebra 20 (1992), 2323--2336.


\bibitem[CRS]{CRS} C. Ciliberto, F. Russo and A. Simis,
{\it  Homaloidal hypersurfaces and hypersurfaces with vanishing Hessian}, Adv. in Math.
 218 (2008), 1759-1805.
 
\bibitem[CoCoA]{CoCoA} J. Abbott, A. M. Bigatti, L. Robbiano,
CoCoA: a system for doing Computations in Commutative Algebra.
Available at http://cocoa.dima.unige.it 

\bibitem[DiPo]{DiPo} L. Di Biagio and E. Postinghel, {\it Apolarity, Hessian  and Macaulay polynomials}, Comm. Algebra  41 (1) (2013), 226-237.

\bibitem[D]{D} A. Dimca, {\it  On the dual and Hessian  mappings of projective 
hypersurfaces}, Math. Proc. Cambridge Phil. Soc. 101 (1987), 461-468.

 


\bibitem[DP]{DP} A. Dimca and D. Popescu, {\it Hilbert series and Lefschetz properties of dimension one almost complete intersections}, Comm. Algebra 44 (2016), 4467--4482.

\bibitem[DSa]{DSa} A. Dimca and M. Saito, {\it A generalization of Griffiths' theorem on rational integrals}, Duke Math. J. 135 (2006), 303--326.

\bibitem[DS]{DS14} A. Dimca and E. Sernesi,  {\it Syzygies and logarithmic vector fields along plane curves},
\textit{Journal de l'\'Ecole polytechnique-Math\'ematiques} 1 (2014), 247-267.

\bibitem[DSt1]{DSt1} A. Dimca and G. Sticlaru, {\it Hessian ideals of a homogeneous polynomial and generalized Tjurina
algebras}, Documenta Math. 20 (2015) 689--705.

\bibitem[DSt2]{DStJump} A. Dimca and G. Sticlaru, {\it On the jumping lines of logarithmic vector fields along plane curves}, arXiv: 1804.06349.

 \bibitem[GR]{GR} A. Garbagnati and F. Repetto, {\it A geometrical approach to
Gordan--Noether's and Franchetta's contributions to a question posed by Hesse}, Collect. Math.
60 (2009), 27--41.



\bibitem[Go]{Go} R. Gondim, {\it On higher Hessians and the Lefschetz properties},  J. Algebra 489 (2017), 241--263. 

\bibitem[GRu]{GRu} R. Gondim and F. Russo, {\it Cubic hypersurfaces with vanishing Hessian}, J. of Pure and Appl. Algebra  219 (2015), 779-806. 

\bibitem[GZ]{GZ} R. Gondim and G. Zappalà , {\it Lefschetz properties for Artinian Gorenstein algebras presented by quadrics} Proc. Amer. Math. Soc. 146 (2018), no. 3, 993--1003. 

\bibitem[GZ2]{GZ2} R. Gondim and G. Zappalà , {\it On mixed Hessians and the Lefschetz properties}, J. of Pure and Appl. Algebra 223 (2019), 4268--4282.

\bibitem[GN]{GN} P. Gordan and M. Noether,
{\it Ueber die algebraischen Formen, deren Hesse'sche Determinante
identisch verschwindet}, Math. Ann. 10 (1876), 547--568.

\bibitem[Gr]{Gr} P. Griffiths, {\it On the periods of certain rational integrals}, I, II. Annals of Mathematics  90 (1969), 460-495 and 496-541.

\bibitem[GH]{GH} P. Griffiths and J. Harris, {\it Algebraic geometry and local differential geometry}, Annales scientifiques de l'Ecole Normale Sup\' erieure 12 (1969), 355--452.




 


\bibitem[HMNW]{HMNW} T. Harima, J. Migliore, U. Nagel and J. Watanabe,   {\it The weak and strong Lefschetz properties for artinian K-algebras}, J. Algebra 262 (2003), 99--126.




\bibitem[He]{He} O. Hesse, {\it
Zur Theorie der ganzen homogenen Functionen}, J. reine angew.
Math. 56 (1859), 263--269.



\bibitem[H]{KH}  K. Hulek, {\it  Stable rank 2 vector bundles on $\p^2$ with $c_1$ odd}, Math. Ann. 242 (1979), 241--266.



\bibitem[Ik]{Ik} H. Ikeda, {\it Results  on  Dilworth  and  Rees  numbers  of  Artinian  local  rings}, Japan J. Math. 22 (1996), 147--158

\bibitem[IG]{IG} G. Ilardi, {\it Jacobian ideals, Arrangements and Lefschetz properties}, J. Algebra 508 (2018), 418--430. 

\bibitem[MW]{MW} T.  Maeno and J.  Watanabe, {\it Lefschetz  elements of  artinian  Gorenstein algebras and Hessians of homogeneous polynomials}, Illinois J. Math. 53 (2009), 593--603.	



\bibitem[MV]{MaVa} S. Marchesi and J. Vall\` es, {\it  Nearly free curves and arrangements: a vector bundle point of view}, arXiv:1712.04867.

\bibitem[MY]{MY} J.N. Mather and S. S.-T. Yau, {\it Classification of isolated hypersurface singularities by their moduli algebras}, Invent. Math. 69 (1982), 243--251.

\bibitem[MMR]{MMR} M. Michalek and R. M. Mir\` o-Roig,  {\it Smooth monomial Togliatti systems of cubics}. Journal of Combinatorial Theory, Series A, 143 (2016), 66--87.

\bibitem[MMN]{MMN} J. Migliore, R. M. Mir\` o-Roig and U. Nagel. {\it Monomial ideals, almost complete intersections and the weak Lefschetz property}, Trans. of the Ame. Math. Soc. 363 (2011), 229--257.

\bibitem[MN1]{MN1} J. Migliore, U. Nagel {\it Survey article: a tour of the weak and strong Lefschetz properties},  J. Commut. Algebra 5 (2013), no. 3, 329--358. 







\bibitem[Ru]{Ru} F. Russo, {\it On the Geometry of Some Special Projective Varieties}, Lecture Notes of the Unione Matematica Italiana, vol. 18 Springer (2016).

\bibitem[Se]{Se} E. Sernesi, {\it  The local cohomology of the jacobian ring}, Documenta Math. 19 (2014), 541-565. 



\bibitem[St2]{St} R. Stanley, {\it Hilbert functions of graded algebras}, Adv. in Math. 28 (1978), 57--83.




\bibitem[Z]{Z} F. L., Zak,   {\it Structure of Gauss maps}, Functional Analysis and Its Applications, 21 (1987), 32--41.

\bibitem[ZW1]{ZW1} Z. Wang, {\it On homogeneous polynomials determined by their Jacobian ideal}, Manuscripta Math. 146 (2015), 559--574.



\bibitem[ZW2]{ZW2} Z. Wang, {\it On homogeneous polynomials determined by their higher jacobians}, arXiv: 1809.02422.

\bibitem[Wa1]{Wa1} J. Watanabe, {\it A remark on the Hessian of homogeneous polynomials}, in {\it The Curves
Seminar at Queen's},  Volume XIII, Queen's Papers in Pure and Appl. Math. 119 (2000), 171--178.

\bibitem[Wa2]{Wa2} J. Watanabe, {\it On the Theory of Gordan-Noether on Homogeneous Forms with Zero Hessian}, Proc. Sch. Sci. TOKAI UNIV. 49 (2014), 1--21.  

\bibitem[Wa3]{Wa3} J. Watanabe, {\it The Dilworth number of Artinian rings and finite posets with rank function}. In: Commutative algebra and combinatorics (pp. 303-312). Mathematical Society of Japan. 1987.

\end{thebibliography}
\end{document}